\documentclass[12pt, letterpaper]{amsart}
\usepackage[dvipsnames, table, xcdraw]{xcolor}
\usepackage{amsfonts, amsthm, mathtools}
\usepackage[all, cmtip]{xy}
\usepackage{fullpage}
\usepackage{bm}
\usepackage{amsmath}
\usepackage{amssymb}
\usepackage{enumerate}
\usepackage{tikz}
\usepackage{hyperref}
\usepackage{cleveref}
\usepackage{verbatim}
\usepackage{cancel}
\usepackage{float}
\usepackage{todonotes}
\usepackage{kotex}
\usepackage{longtable}
\usepackage{algorithm, algpseudocode}
\usepackage{multirow}

\usepackage{graphicx} 

\newcommand{\C}{\mathbb{C}}

\newcommand{\HH}{\mathbb{H}}
\newcommand{\Z}{\mathbb{Z}}

\newcommand{\SL}{\operatorname{SL}}

\newtheorem{lem}{Lemma}[section]
\newtheorem{thm}[lem]{Theorem}
\Crefrangeformat{thm}{Theorems #3#1#4-#5#2#6}

\Crefrangeformat{prop}{Propositions #3#1#4-#5#2#6}
\newtheorem{cor}[lem]{Corollary}

\theoremstyle{definition}
\newtheorem{remark}[lem]{Remark}
\newtheorem{definition}[lem]{Definition}

\newcommand{\github}[1]{\href{https://github.com/kwon314159/Explicit_constructions_of_cyclic_N_isogenies/blob/5495d0b472eb6dfdd3325fbc35e2890a0d2235d6/#1}{\path{#1}}}

\title{Explicit constructions of cyclic $N$-isogenies}

\begin{document}
	\begin{abstract}
		The modular curve $X_{0}(N)$ parametrizes elliptic curves together with a cyclic
		subgroup of order~$N$, and hence cyclic $N$-isogenies. While explicit moduli
		descriptions of $X_{1}(N)$ are well developed, a comparable construction for
		$X_{0}(N)$ has remained incomplete. We give a uniform method for constructing
		explicit generators of $\mathbb{C}(X_{0}(N))$, extending an approach of Dowd~\cite{Dowd},
		and use them to obtain a concrete moduli interpretation of cyclic $N$-isogenies.
		This yields explicit formulas for sporadic rational points on $X_{0}(N)$ and
		the associated isogenies, providing a unified solution to the moduli problem
		for $X_{0}(N)$.
	\end{abstract}


	\thanks{The authors were supported by Basic Science Research Program through
	the National Research Foundation of Korea (NRF) funded by the Ministry of
	Education (No. 2022R1A2C1010487).}

	\author{\sc Daeyeol Jeon}
	\address{Daeyeol Jeon \\ Department of Mathematics Education \\ Kongju National University \\ Gongju, 32588 South Korea}
	\email{dyjeon@kongju.ac.kr}

	\author{\sc Yongjae Kwon}
	\address{Yongjae Kwon \\ Department of Mathematics Education \\ Kongju National University \\ Gongju, 32588 South Korea}
	\email{211049@kongju.ac.kr}

	\maketitle

	\section{Introduction}

	The modular curves $X_{1}(N)$ and $X_{0}(N)$ occupy a central place in the
	arithmetic theory of elliptic curves. The curve $X_{1}(N)$ (with cusps removed)
	parametrizes elliptic curves equipped with a point of exact order~$N$, while
	$X_{0}(N)$ (with cusps removed) parametrizes elliptic curves together with a
	cyclic subgroup of order~$N$. Since every cyclic subgroup $C\subset E$ of order~$N$
	arises as the kernel of a cyclic $N$-isogeny $E\to E'$, parametrizing pairs
	$(E,C)$ is equivalent to parametrizing the associated cyclic $N$-isogenies.

	A substantial body of work---including contributions of Baaziz~\cite{Baaziz}, Dowd~\cite{Dowd},
	Galbraith~\cite{Galbraith-thesis}, Reichert~\cite{Reichert} Sutherland~\cite{sutherland2},
	Tsukazaki~\cite{Tsukazaki}, and Yang~\cite{Yang}---has produced explicit defining
	equations for $X_{0}(N)$ and $X_{1}(N)$. Several of these works go beyond the
	determination of equations by offering explicit methods for solving the
	associated moduli problems, that is, recovering elliptic curves (and their level
	structures) directly from points on the corresponding modular curves.

	For $X_{1}(N)$, Sutherland~\cite{sutherland2} constructed low-degree models by
	exploiting the Tate normal form and enforcing the condition that $(0,0)$ have exact
	order~$N$, thereby giving an efficient and uniform solution to its moduli
	problem. Baaziz~\cite{Baaziz} later provided a more conceptual approach based on
	the Weierstrass $\wp$-function, showing that the $N$-division values of $\wp$
	and $\wp'$ furnish a general mechanism for recovering the elliptic curves
	associated to points on $X_{1}(N)$.

	For $X_{0}(N)$, progress has been more sporadic. In special situations, explicit
	descriptions of cyclic $N$-isogenies yield solutions to its moduli problem.
	When $X_{0}(N)$ has genus~0, or when the Fricke quotient $X_{0}^{+}(N)$ has genus~0,
	explicit parametrizations were obtained by Cremona, Cooley, and collaborators,
	leading to complete models for the associated isogenies in these low-genus cases.
	More recently, Dowd~\cite{Dowd} introduced a systematic framework that uniformly
	treats all genus–0 levels, both for $X_{0}(N)$ and for $X_{0}^{+}(N)$, by expressing
	cyclic $N$-isogenies in terms of carefully chosen modular functions.

	A key observation that emerges from Dowd’s work is that the modular functions he
	uses in the genus–0 cases in fact generate the entire function field $\C(X_{0}(
	N))$. Once this is recognized, it becomes clear that the same philosophy can be
	pushed much further: knowing explicit generators of $\C(X_{0}(N))$ provides
	enough information to reconstruct the elliptic curves and their $N$-cyclic
	isogenies from arbitrary points of $X_{0}(N)$, thereby supplying a conceptual route
	to solving the moduli problem in complete generality.

	The small-degree isogeny tables of the Sage documentation \cite{SageIsogenySmallDegree}
	highlight related work of Cremona and Cooley, who constructed cyclic $N$-isogenies
	not only for the genus~0 cases but also for the sporadic rational points on
	$X_{0}(N)$. Independently, Barrios~\cite{Barrios} obtained analogous
	parametrizations for the sporadic points.

	\medskip

The purpose of this paper is to extend these ideas beyond the low-genus and
sporadic cases. Using explicit generators of the function field
$\mathbb{C}(X_{0}(N))$, we construct cyclic $N$-isogenies from arbitrary points
of $X_{0}(N)$. This yields a complete solution to the moduli problem for
$X_{0}(N)$, without restrictions on the genus. Our approach unifies previously
known constructions within a single framework and applies equally well to
existing models of $X_{0}(N)$ in the literature. In particular, although the
defining equations obtained by Galbraith \cite{Galbraith-thesis} and Yang \cite{Yang} are not accompanied by a
moduli interpretation, our method recovers the associated cyclic
$N$-isogenies directly from their points. We illustrate the effectiveness of
this approach through explicit computations in several sporadic cases.

The results in this paper were obtained using SageMath \cite{sagemath}, and
the codes used to verify all computations are available at
\begin{center}
  \url{https://github.com/kwon314159/Explicit_constructions_of_cyclic_N_isogenies}
\end{center}

	\section{Preliminaries}

	In this section we recall the moduli interpretation of the modular curves $X_{0}
	(N)$ and $X_{1}(N)$, and introduce the Eisenstein series required later for
	writing defining equations of $X_{0}(N)$. Throughout, $\HH$ denotes the
	complex upper half–plane and $q=e^{2\pi i\tau}$.

	\subsection{Modular curves and their moduli problems}

	Let $\Gamma_{0}(N)$ and $\Gamma_{1}(N)$ be the congruence subgroups
	\begin{align*}
		\Gamma_{0}(N)= & \left\{ \begin{pmatrix}a & b \\ c & d\end{pmatrix}\in \SL_{2}(\Z):\; c\equiv 0 \pmod N \right\},                     \\
		\Gamma_{1}(N)= & \left\{ \begin{pmatrix}a & b \\ c & d\end{pmatrix}\in \SL_{2}(\Z):\; c\equiv 0,\ a\equiv d\equiv 1 \pmod N \right\}.
	\end{align*}

	The modular curves $Y_{0}(N)$ and $Y_{1}(N)$ are defined as the quotients
	\[
		Y_{0}(N)=\Gamma_{0}(N) \backslash \HH, \qquad Y_{1}(N)=\Gamma_{1}(N) \backslash
		\HH,
	\]
	and their compactifications $X_{0}(N)$, $X_{1}(N)$ are obtained by adjoining cusps.

	\begin{definition}
		An \emph{enhanced elliptic curve of $\Gamma_{0}(N)$–type} is a pair $(E,C)$,
		where $E/\C$ is an elliptic curve and $C\subset E$ is a cyclic subgroup of order
		$N$. Two such pairs $(E,C)\sim(E',C')$ if there exists an isomorphism
		$\varphi:E\to E'$ satisfying $\varphi(C)=C'$.
	\end{definition}

	\begin{definition}
		An \emph{enhanced elliptic curve of $\Gamma_{1}(N)$–type} is a pair $(E,P)$,
		where $E/\C$ is an elliptic curve and $P\in E$ is a point of exact order $N$.
		Two such pairs $(E,P)\sim(E',P')$ if there exists an isomorphism
		$\varphi:E\to E'$ satisfying $\varphi(P)=P'$.
	\end{definition}

	The fundamental moduli theorem (see e.g.\cite[Theorem~2.1.8]{modular}) states:

	\begin{thm}
		There are natural bijections
		\[
			Y_{0}(N)(\C)\;\longleftrightarrow\;\{(E,C)\}/\sim, \qquad Y_{1}(N)(\C)\;\longleftrightarrow
			\;\{(E,P)\}/\sim.
		\]
		Thus $Y_{0}(N)$ parameterizes elliptic curves equipped with a cyclic subgroup
		of order $N$, while $Y_{1}(N)$ parameterizes elliptic curves equipped with a
		point of order $N$.
	\end{thm}

	In particular, a point $\tau\in Y_{0}(N)$ corresponds to the cyclic subgroup $C
	_{\tau}=\langle 1/N\rangle\subset \C/\Lambda_{\tau}$, and hence to the cyclic
	$N$–isogeny
	\[
		\C/\Lambda_{\tau} \longrightarrow \C/\Lambda_{N\tau},\quad z\mapsto Nz,
	\]
	where $\Lambda_{\tau}=\Z+\tau\Z$.

	\subsection{Eisenstein series and defining equations}

	For even $k\ge 4$, the normalized Eisenstein series is defined by
	\[
		E_{k}(\tau)=1-\frac{2k}{B_{k}}\sum_{n\ge 1}\sigma_{k-1}(n)q^{n},
	\]
	where $B_{k}$ is the $k$-th Bernoulli number and $\sigma_{k-1}(n)=\sum_{d\mid n}
	d^{k-1}$. These are holomorphic modular forms of weight $k$ on $\SL_{2}(\Z)$ (cf.
	\cite{modular}).

	It is customary to set
	\[
		G_{k}(\tau)=2\zeta(k)\,E_{k}(\tau),
	\]
	which appear naturally in the analytic uniformization of elliptic curves.

	The Weierstrass model associated with the lattice $\Lambda_{\tau}$ is
	\[
		E_{\tau}:\quad y^{2}=4x^{3}-60G_{4}(\tau)x-140G_{6}(\tau),
	\]
	and every complex elliptic curve arises in this way (Uniformization Theorem, \cite[Proposition
	VI.3.6 and Corollary VI.5.1.1]{silverman}).

	For arithmetic applications it is convenient to use the rescaled forms
	\[
		a_{4}(\tau)=-\frac{E_{4}(\tau)}{48},\qquad a_{6}(\tau)=\frac{E_{6}(\tau)}{864}
		,
	\]
	so that an elliptic curve may be written as
	\[
		E_{\tau}:\quad y^{2}=x^{3}+a_{4}(\tau)x+a_{6}(\tau).
	\]

	For later use we recall the standard Eisenstein series of weight $2$ and level
	$N$,
	\[
		E_{2}^{(N)}(\tau) :=\frac{1}{2\pi i}\frac{d}{d\tau}\log\!\left( \frac{\eta(N\tau)}{\eta(\tau)}
		\right) = \frac{N-1}{24}+\sum_{n\ge1}\sigma_{1}(n)\bigl(q^{n} - N q^{Nn}\bigr
		),
	\]
	where $\eta$ is the Dedekind eta function. Then $E_{2}^{(N)}$ is a modular form
	of weight $2$ on $\Gamma_{0}(N)$ and is anti–invariant under the Fricke involution:
	\[
		E_{2}^{(N)}\left(\frac{-1}{N\tau}\right) = -N\tau^{2}\, E_{2}^{(N)}(\tau).
	\]

	\subsection{Genus $0$ case and rational construction of cyclic $N$–isogenies}

	When the modular curve $X_{0}(N)$ has genus $0$, its function field is generated
	by a single Hauptmodul $h_{N}$, i.e.\
	\[
		\C(X_{0}(N))=\C(h_{N}), \qquad h_{N} = q^{-1}+ O(1).
	\]
	Dowd~\cite{Dowd}, §3.3, uses this fact to give \emph{explicit rational
	formulas} for the coefficients of elliptic curves in a cyclic $N$–isogeny pair.

	More precisely, let
	\[
		E_{\tau}:\quad y^{2} = x^{3} + a_{4}(\tau)x + a_{6}(\tau), \qquad a_{4}(\tau)
		=-\frac{E_{4}(\tau)}{48\,E_{2}^{(N)}(\tau)^{2}},\quad a_{6}(\tau)=\frac{E_{6}(\tau)}{864\,E_{2}^{(N)}(\tau)^{3}}
		,
	\]
	be the model obtained from the Tate curve by taking
	$\lambda(\tau)=E_{2}^{(N)}(\tau)$. The isogenous curve corresponding to the dual
	point $N\tau$ on $Y_{0}(N)$ is then
	\[
		E_{N\tau}:\quad y^{2} = x^{3} + a'_{4}(\tau)x + a'_{6}(\tau), \qquad a'_{4}(\tau
		)=-\frac{E_{4}(N\tau)}{48\,E_{2}^{(N)}(\tau)^{2}},\quad a'_{6}(\tau)=\frac{E_{6}(N\tau)}{864\,E_{2}^{(N)}(\tau)^{3}}
		.
	\]

	Since $X_{0}(N)$ has genus $0$, both $a_{4}$, $a_{6}$ and their transforms
	$a'_{4}$, $a'_{6}$ can be expressed as rational functions of $h_{N}$ by comparing
	$q$–expansions. Dowd carries this out explicitly, obtaining rational
	expressions
	\[
		a_{4} = A_{4}(h_{N}),\qquad a_{6} = A_{6}(h_{N}),\qquad a'_{4} = A'_{4}(h_{N}
		),\qquad a'_{6} = A'_{6}(h_{N}),
	\]
	where each right–hand side is a rational function in $h_{N}$ with integer
	coefficients.

	Thus, when $X_{0}(N)$ has genus $0$, \emph{the entire cyclic $N$–isogeny}
	\[
		E \longrightarrow E'
	\]
	is described explicitly by rational functions in the single parameter $h_{N}$.
	This provides a uniform algebraic construction of all cyclic $N$–isogenies parameterized
	by $X_{0}(N)$, and is the basis for the tables of isogeny coefficients
	appearing in Dowd’s work.

	\section{Constructing cyclic $N$-isogenies}

	As explained in Section 2.3, when $X_{0}(N)$ has genus $0$ the modular
	functions $a_{4}$ and $a_{6}$ defined by
	\[
		a_{4}(\tau)=-\frac{E_{4}(\tau)}{48\,E_{2}^{(N)}(\tau)^{2}},\qquad a_{6}(\tau)
		=\frac{E_{6}(\tau)}{864\,E_{2}^{(N)}(\tau)^{3}}
	\]
	admit explicit expressions as rational functions of a Hauptmodul, and this suffices
	to solve the moduli problem for cyclic $N$-isogenies.

	This phenomenon does not depend solely on the rationality of $X_{0}(N)$;
	rather, Dowd's insightful choice of the modular functions $a_{4}$ and $a_{6}$ points
	to a principle of broader validity. These functions in fact generate the
	function field of $X_{0}(N)$, which allows the moduli problem to be treated uniformly
	for all $N$.

	In this section, we establish that $(a_{4},a_{6})$ separates the points of
	$Y_{0}(N)$, which in turn implies that they generate the function field $\C(X_{0}
	(N))$.

	\subsection{Generators of the function field}

	We first show that the map defined by these invariants is injective, thereby embedding
	the modular curve into its image in $\C^{2}$.

	\begin{thm}
		\label{thm:bijection} Let \(Y_{0}'(N) \subset Y_{0}(N)\) denote the set of points \([\tau]\) such that
\(E_{2}^{(N)}(\tau)\neq 0\). Then, the map $\Phi: Y_{0}^{\prime}(N) \to \C^{2}$ defined by $[\tau]
		\mapsto (a_{4}(\tau), a_{6}(\tau))$ is well-defined and injective. That is, for
		any $\tau, z \in \HH$,
		\[
			a_{4}(\tau) = a_{4}(z) \quad \text{and}\quad a_{6}(\tau) = a_{6}(z) \iff z
			= \gamma \tau \quad \text{for some }\gamma \in \Gamma_{0}(N).
		\]
	\end{thm}

	\begin{proof}
		First, we verify that $a_{4}$ and $a_{6}$ are well-defined on
		$Y_{0}^{\prime}(N)$. Recall that a point on $Y_{0}(N)$
		represents an equivalence class of pairs $(E, C)$. An isomorphism between
		elliptic curves (a twist) corresponds to a scaling of the lattice
		$\Lambda_{\tau} \to \lambda \Lambda_{\tau}$. Under this scaling, the
		Eisenstein series transform as $E_{k} \to \lambda^{-k}E_{k}$. Crucially, for
		$\gamma \in \Gamma_{0}(N)$, the function $E_{2}^{(N)}$ transforms as a modular
		form of weight 2, i.e.,
		$E_{2}^{(N)}(\gamma\tau) = (c\tau+d)^{2} E_{2}^{(N)}(\tau)$. If we set $\lambda
		= (c\tau+d)^{-1}$, this means the twist factor $E_{2}^{(N)}$ scales
		exactly as required to cancel the weights of $E_{4}$ and $E_{6}$. Explicitly,
		\[
			a_{4}(\gamma\tau) = -\frac{(c\tau+d)^{4} E_{4}(\tau)}{48 ((c\tau+d)^{2} E_{2}^{(N)}(\tau))^{2}}
			= a_{4}(\tau),
		\]
		and similarly $a_{6}(\gamma\tau) = a_{6}(\tau)$. This confirms that $a_{4}$ and
		$a_{6}$ are invariant under the action of $\Gamma_{0}(N)$.

		Conversely, suppose that $a_{4}(\tau) = a_{4}(z)$ and
		$a_{6}(\tau) = a_{6}(z)$. The equality of these invariants implies the equality
		of the $j$-invariant, since $j$ is a rational function of $a_{4}$ and
		$a_{6}$. Since $j(\tau) = j(z)$, there exists $\gamma =
		\begin{pmatrix}
			a & b \\
			c & d
		\end{pmatrix}
		\in \SL_{2}(\Z)$ such that $z = \gamma \tau$. We must show that $\gamma \in \Gamma
		_{0}(N)$.

		Substituting $z = \gamma \tau$ into the relations $a_{4}(z) = a_{4}(\tau)$ and
		$a_{6}(z) = a_{6}(\tau)$, and using the modularity of $E_{4}$ and $E_{6}$, we
		obtain:
		\[
			\frac{(c\tau+d)^{4} E_{4}(\tau)}{E_{2}^{(N)}(\gamma\tau)^{2}}= \frac{E_{4}(\tau)}{E_{2}^{(N)}(\tau)^{2}}
			\quad \text{and}\quad \frac{(c\tau+d)^{6} E_{6}(\tau)}{E_{2}^{(N)}(\gamma\tau)^{3}}
			= \frac{E_{6}(\tau)}{E_{2}^{(N)}(\tau)^{3}}.
		\]
		Since the discriminant $\Delta(\tau) \neq 0$, the values $E_{4}(\tau)$ and $E
		_{6}(\tau)$ cannot vanish simultaneously. In either case, we deduce that
		\[
			E_{2}^{(N)}(\gamma\tau) = \epsilon (c\tau+d)^{2} E_{2}^{(N)}(\tau)
		\]
		for some root of unity $\epsilon$. This implies that $E_{2}^{(N)}$
		transforms as a modular form of weight 2 under $\gamma$. Recall the quasi-modular
		transformation law for $E_{2}$:
		\[
			E_{2}(\gamma\tau) = (c\tau+d)^{2} E_{2}(\tau) + \frac{6c(c\tau+d)}{\pi i},
			\quad \text{for }\gamma \in \SL_{2}(\Z).
		\]

		Consequently, the linear combination
		$E_{2}^{(N)}(\tau) = E_{2}(\tau) - N E_{2}(N\tau)$ eliminates the non-modular
		term if and only if $c \equiv 0 \pmod N$. Therefore, we must have
		$\gamma \in \Gamma_{0}(N)$.
	\end{proof}

	It is a fundamental result in the theory of Riemann surfaces that a subset $S$
	of the function field $\C(X)$ generates $\C(X)$ if and only if there exists a non-empty
	open subset $U \subseteq X$ on which $S$ separates points. Explicitly, this condition
	means that for any two points $P, Q \in U$, if $\phi(P) = \phi(Q)$ for all
	$\phi \in S$, then $P = Q$. \Cref{thm:bijection} establishes this injectivity
	for the pair $S = \{a_{4}, a_{6}\}$ on the open subset
	$Y_{0}'(N) \subset X_{0}(N)$, since $Y_0'(N)$ is obtained from $X_0(N)$ by removing finitely many
    points, namely the cusps and the zeros of $E_2^{(N)}(\tau)$. 
    Consequently, these functions constitute a
	generating set for the function field.

	\begin{cor}
		\label{cor:generators} The function field of the modular curve $X_{0}(N)$ is
		generated by $a_{4}$ and $a_{6}$. That is,
		\[
			\C(X_{0}(N)) = \C(a_{4}, a_{6}).
		\]
	\end{cor}

	By Corollary~\ref{cor:generators}, the algebraic relation satisfied by $a_{4}$
	and $a_{6}$ yields a plane defining equation for $X_{0}(N)$. However, much
	like the modular equation relating $j(\tau)$ and $j(N\tau)$, this equation typically
	has very high degree and extremely large coefficients, making it rather
	impractical for solving the associated moduli problem. For instance, already
	in the lowest level of positive genus, $N=11$, the defining equation takes the
	following form:

	\begin{align}
		\label{eq:11} & -29241x^{6} - 23955822x^{5} - 1351692x^{4}y + 572544x^{3}y^{2} - 15183229435x^{4}    \\
		\notag        & \quad + 7092313360x^{3}y - 1934162736x^{2}y^{2} + 235016704xy^{3} - 10061824y^{4}    \\
		\notag        & \quad + 103990630700x^{3} - 301970625000x^{2}y + 47640642720xy^{2} - 4119072320y^{3} \\
		\notag        & \quad - 2009614509375x^{2} + 2923075650000xy - 2204530508400y^{2}                    \\
		\notag        & \quad + 1296871230050x - 5894869227500y + 285311670611=0.
	\end{align}

	Because of this complexity, we do not rely on this model directly. Instead, we
	employ the models developed by Galbrath
	\cite{Galbraith-thesis} and Yang \cite{Yang}, or use a plane equation defined by suitable rational functions
	in cuspforms of weight 2 for $\Gamma_{0}(N)$.

	We now explain how the moduli problem for $X_{0}(N)$ can be solved in the case
	$N=11$. We begin by constructing the cyclic $N$-isogeny
	$E_{\tau} \to E_{N\tau}$ using the $q$-expansions of the invariants
	$a_{4}, a_{6}, a_{4}', a_{6}'$. This can be carried out by the following procedure:
	\begin{algorithm}
		\caption{Explicit Construction of Cyclic $N$-Isogeny Pairs}
		\label{al:isogeny}
		\begin{algorithmic}
			[1] \Require Level $N$, generators $X, Y$ of the function field
			$\mathbb{Q}(X_{0}(N))$, and a non-cuspidal point $P \in X_{0}(N)(\mathbb{Q}
			)$. \Ensure The pair of elliptic curves $(E, E')$ linked by a cyclic $N$-isogeny.

			\State Define the modular functions corresponding to the Weierstrass
			coefficients of $E$ and $E'$:
			\[
				\begin{aligned}
					a_{4}  & = \frac{-E_4(\tau)}{48 (E_2^{(N)}(\tau))^2},  & a_{6}  & = \frac{E_6(\tau)}{864 (E_2^{(N)}(\tau))^3},  \\
					a_{4}' & = \frac{-E_4(N\tau)}{48 (E_2^{(N)}(\tau))^2}, & a_{6}' & = \frac{E_6(N\tau)}{864 (E_2^{(N)}(\tau))^3}.
				\end{aligned}
			\]

			\State Express these invariants as rational functions in the coordinate
			ring of $X_{0}(N)$:
			\[
				a_{4} = \mathcal{A}(X, Y), \quad a_{6} = \mathcal{B}(X, Y) \quad \text{and}
				\quad a_{4}' = \mathcal{A}'(X, Y), \quad a_{6}' = \mathcal{B}'(X, Y).
			\]

			\State Evaluate the rational functions at the point $P$:
			\[
				(A, B) \leftarrow (\mathcal{A}(P), \mathcal{B}(P)) \quad \text{and}\quad
				(A', B') \leftarrow (\mathcal{A}'(P), \mathcal{B}'(P)).
			\]

			\State \Return the pair of curves $(E, E')$:
			\[
				E: y^{2} = x^{3} + Ax + B, \qquad E': y^{2} = x^{3} + A'x + B'.
			\]
		\end{algorithmic}
	\end{algorithm}

	\noindent
	Here we remark that the second step of Algorithm \ref{al:isogeny}, which
	consists in expressing these invariants as rational functions in the
	coordinate ring of $X_{0}(N)$, is carried out following \cite[Algorithm~4.3]{Jeon2025}.

	In the next step, we choose a ``good'' model of $X_{0}(N)$, i.e., one whose defining
	equation has relatively small degree and manageable coefficients, such as the
	model constructed by Yang \cite{Yang}. Once such a model is fixed, we
	determine the points on $X_{0}(N)$ that arise in the correspondence relevant
	to the moduli problem; depending on the level, this may include computing its
	$\mathbb{Q}$-rational points when they exist, as in the case of $X_{0}(11)$.
	For illustration, Yang's model of $X_{0}(11)$ is given by the equation
	\[
		X_{0}(11):\qquad Y^{2} + Y = X^{3} - X^{2} - 10X - 20,
	\]
	and the corresponding $q$-expansions of the functions $X$ and $Y$ are
	\[
		\begin{aligned}
			X & = q^{-2}+ 2q^{-1}+ 4 + 5q + 8q^{2} + q^{3} + 7q^{4} - 11q^{5} + \cdots,                \\
			Y & = q^{-3}+ 3q^{-2}+ 7q^{-1}+ 12 + 17q + 26q^{2} + 19q^{3} + 37q^{4} - 15q^{5} - \cdots.
		\end{aligned}
	\]

	The curve $X_{0}(11)$ has exactly five $\mathbb{Q}$-rational points, of which three
	are non-cuspidal:
	\[
		(5,5),\qquad (5,-6),\qquad (16,-61).
	\]
	These points correspond to the isomorphism classes of elliptic curves equipped
	with $11$-cyclic isogenies, and, since they are $\mathbb{Q}$-rational, the corresponding
	elliptic curves and isogenies admit $\mathbb{Q}$-rational models.

	Next, we determine the cyclic $11$-isogenies corresponding to the three
	$\mathbb{Q}$-rational points on $X_{0}(11)$. Recall that the generators of our
	model, denoted by $x = a_{4}$ and $y = a_{6}$, have the following $q$-expansions:
	\[
		\begin{aligned}
			x & = -\frac{3}{25}- \frac{3528}{125}q - \frac{75816}{625}q^{2} + \frac{1097856}{3125}q^{3} + \frac{593496}{3125}q^{4} - \frac{106231824}{78125}q^{5}+ \cdots,        \\
			y & = \frac{2}{125}- \frac{5112}{625}q - \frac{649512}{3125}q^{2} - \frac{485856}{3125}q^{3} + \frac{229154456}{15625}q^{4} + \frac{634190256}{390625}q^{5}+ \cdots .
		\end{aligned}
	\]

	Using Algorithm \ref{al:isogeny}, we can express our generators as rational
	functions in Yang's generators. This yields a rational map from Yang's model to
	our model, expressed as:
	\[
		x =\mathcal{A}(X,Y)= \frac{A_{Y}(X)Y + A_{X}(X)}{Q(X)^{2}}, \qquad y =\mathcal{B}
		(X,Y) = \frac{B_{Y}(X)Y + B_{X}(X)}{Q(X)^{3}},
	\]
	where the common factor in the denominators is the quadratic polynomial
	\[
		Q(X) = 25X^{2} + 86X + 89.
	\]
	The numerator polynomials are given by:
	{\allowdisplaybreaks \begin{align*}A_{Y}(X)&= -17640 X^{2} - 106344 X - 107568, \\ A_{X}(X)&= -75 X^{4} - 93972 X^{3} - 445362 X^{2} - 881916 X - 738867, \\ B_{Y}(X)&= -127800 X^{4} - 10626696 X^{3}- 51849288 X^{2}- 85057272 X - 45566928, \\ B_{X}(X)&= 250 X^{6}- 3372780 X^{5}- 33335514 X^{4}- 136910656 X^{3}\\&\quad - 317360754 X^{2}- 408243108 X - 220844302.\end{align*} }
	Under this map, the $\mathbb{Q}$-rational points listed above correspond to the
	following $\mathbb{Q}$-rational points on our model:

	\[
		\left(-\frac{4323}{169}, -\frac{109406}{2197}\right), \qquad \left(-\frac{33}{2}
		, -\frac{847}{32}\right), \qquad \left(-\frac{363}{169}, -\frac{10406}{2197}\right
		).
	\]

	Similarly, we derived the explicit rational functions for the invariants of the
	codomain curve $E'$, denoted by $a_{4}'=\mathcal{A}'(X,Y)$ and
	$a_{6}'=\mathcal{B}'(X,Y)$. Evaluating these functions at the rational points allows
	us to construct the explicit models for the isogenies over $\mathbb{Q}$.

	For instance, substituting the point $P=(5,5)$ into our expressions yields the
	isogeny pair:
	\[
		E: y^{2} = x^{3} - \frac{4323}{169}x - \frac{109406}{2197}\quad \xrightarrow{\phi_{11}}
		\quad E': y^{2} = x^{3} - \frac{3}{169}x + \frac{86}{24167}.
	\]
	The elliptic curves constructed by our algorithm are defined over $\mathbb{Q}$,
	but their coefficients involve fractions due to the canonical normalization of
	the modular generators. Recall that a rational point on $X_{0}(N)$ parametrizes
	the $\overline{\mathbb{Q}}$-isomorphism class of the isogeny, which over
	$\mathbb{Q}$ splits into a family of quadratic twists. Our analysis identifies
	the computed domain curve $E$ as the quadratic twist of the standard minimal curve
	\href{https://www.lmfdb.org/EllipticCurve/Q/121.a1}{\textbf{121a1}} by the factor $D=39$. The codomain curve $E'$ is identified as the
	twist of \href{https://www.lmfdb.org/EllipticCurve/Q/121.c1}{\textbf{121c1}} by the factor $D'=-429$.

	The results for all non-cuspidal rational points are summarized in Table~\ref{tab:x011_results}.
	The notation $\mathbf{L}^{(D)}$ in the table denotes the quadratic twist of the
	curve with Cremona label $\mathbf{L}$ by the factor $D$.
	\begin{table}[h]
		\centering
		\renewcommand{\arraystretch}{1.5}
		\begin{tabular}{c|cc|c}
			\hline
			\textbf{Point} $P$ & \textbf{Domain Curve} $E$                   & \textbf{Codomain Curve} $E'$               & \textbf{Isogeny Class}                                                                                                                                         \\
			$(X, Y)$           & $(a_{4}, a_{6})$                            & $(a_{4}', a_{6}')$                         & $E \to E'$                                                                                                                                                     \\
			\hline
			$(5, 5)$           & $(-\frac{4323}{169}, -\frac{109406}{2197})$ & $(-\frac{3}{169}, \frac{86}{24167})$       & \href{https://www.lmfdb.org/EllipticCurve/Q/121.a1}{\textbf{121a1}}$^{(39)}\to$ \href{https://www.lmfdb.org/EllipticCurve/Q/121.c1}{\textbf{121c1}}$^{(-429)}$ \\
			$(16, -61)$        & $(-\frac{363}{169}, -\frac{10406}{2197})$   & $(-\frac{393}{1859}, \frac{9946}{265837})$ & \href{https://www.lmfdb.org/EllipticCurve/Q/121.c1}{\textbf{121c1}}$^{(39)}\to$ \href{https://www.lmfdb.org/EllipticCurve/Q/121.a1}{\textbf{121a1}}$^{(-429)}$ \\
			$(5, -6)$          & $(-\frac{33}{2}, -\frac{847}{32})$          & $(-\frac{3}{22}, \frac{7}{352})$           & \href{https://www.lmfdb.org/EllipticCurve/Q/121.b1}{\textbf{121b1}}$^{(-486)}\to$ \href{https://www.lmfdb.org/EllipticCurve/Q/121.b1}{\textbf{121b1}}$^{(66)}$ \\
			\hline
		\end{tabular}
		\caption{Explicit $\mathbb{Q}$-rational cyclic 11-isogenies constructed from
		points on $X_{0}(11)$. }
		\label{tab:x011_results}
	\end{table}
	\begin{remark}
		The points $P=(5,5)$ and $P=(16,-61)$ are related by the Fricke involution $w
		_{11}$ on $X_{0}(11)$. Since $w_{11}$ maps an isogeny $\phi: E \to E'$ to
		its dual $\hat{\phi}: E' \to E$, it swaps the isomorphism classes of the domain
		and codomain. Our explicit construction captures this symmetry: the second row
		of Table~\ref{tab:x011_results} is effectively the dual of the first row.
		The third point $P=(5,-6)$ corresponds to an elliptic curve admitting an
		endomorphism of degree 11 defined over $\mathbb{Q}$.
	\end{remark}
	\section{Sporadic Cases}

	In this section, we extend the explicit construction of cyclic $N$-isogenies
	to the finite set of levels $N$ for which the modular curve $X_{0}(N)$ has
	positive genus but admits non-cuspidal $\mathbb{Q}$-rational points:
	\[
		N \in \{ 11, 14, 15, 17, 19, 21, 27, 37, 43, 67, 163 \}.
	\]
	The existence and count of these rational points were established through the celebrated
	work of Mazur for prime levels \cite{mazur78}, and subsequently completed for
	composite levels by Kenku \cite{kenku}, Ligozat, and Oesterl\'{e}
	\cite{oesterle}. These foundational works confirmed that the set of such levels
	is finite and determined the exact number of non-cuspidal rational points,
	denoted by $\nu_{N}$, for each case. The values of $\nu_{N}$ are summarized in
	Table \ref{tab:rational_point_counts}.

	\begin{table}[h!]
		\centering
		\caption{Number of non-cuspidal $\mathbb{Q}$-rational points $\nu_{N}$ on
		$X_{0}(N)$ for sporadic levels.}
		\label{tab:rational_point_counts}
		\renewcommand{\arraystretch}{1.2}
		\begin{tabular}{|c||c|c|c|c|c|c|c|c|c|c|c|}
			\hline
			$N$       & 11 & 14 & 15 & 17 & 19 & 21 & 27 & 37 & 43 & 67 & 163 \\
			\hline
			$\nu_{N}$ & 3  & 2  & 4  & 2  & 1  & 4  & 1  & 2  & 1  & 1  & 1   \\
			\hline
		\end{tabular}
	\end{table}
	Building on these classification results, Barrios \cite{Barrios} recently provided
	an explicit classification of the isogeny graphs for all rational elliptic curves
	admitting non-trivial isogenies. He identified that the non-cuspidal points on
	$X_{0}(N)$ correspond, up to quadratic twist, to specific isogeny classes of
	elliptic curves listed in the LMFDB \cite{lmfdb}.

	For composite levels, the cyclic $N$-isogeny $\phi: E \to E'$ corresponding to
	a rational point on $X_{0}(N)$ generally factors into a composition of isogenies
	of prime degrees defined over $\mathbb{Q}$. This factorization explains the
	multiplicity of rational points observed for certain composite levels. For instance,
	in the case of $N=15$, the underlying elliptic curves admit a $\mathbb{Q}$-rational
	3-isogeny (and its dual) as well as a $\mathbb{Q}$-rational 5-isogeny (and its
	dual). Consequently, the four non-cuspidal rational points on $X_{0}(15)$ arise
	precisely from the four possible compositions of these prime degree maps.

	To explicitly construct these isogenies, we implemented Algorithm
	\ref{al:isogeny} using SageMath \cite{sagemath}. For the coordinate systems of
	the modular curves, we employed the defining equations constructed by Yang
	\cite{Yang} for the levels $N \le 27$, and the canonical models computed by Galbraith
	\cite{Galbraith-thesis} for the higher genus cases $N \in \{37, 43, 67, 163\}$.

	For the specific case of $N=67$, where the curve has genus 5, we utilized a specific
	basis of cuspforms to obtain a system of defining equations distinct from the model
	in \cite{Galbraith-thesis}. We computed a basis $\{x_{0}, \dots, x_{4}\}$ for
	the space $S_{2}(\Gamma_{0}(67))$ with the following $q$-expansions:
	\begin{align*}
		x_{0} & = q^{3} - q^{4} - q^{5} + q^{6} - q^{8} + O(q^{10}),                   \\
		x_{1} & = q^{2} - q^{5} - q^{6} - q^{7} - q^{8} + q^{9} + O(q^{10}),           \\
		x_{2} & = q^{2} - q^{3} - q^{4} - q^{6} + q^{7} + 2q^{8} + 2q^{9} + O(q^{10}), \\
		x_{3} & = q^{2} + q^{3} + 2q^{5} - q^{6} - q^{7} - q^{9} + O(q^{10}),          \\
		x_{4} & = q + 2q^{5} - q^{9} + O(q^{10}).
	\end{align*}
	Using the linear dependencies among the products $x_{i} x_{j}$, we determined
	that the canonical image of $X_{0}(67) \subset \mathbb{P}^{4}$ is defined by
	the intersection of the following three quadrics:
	\begin{align*}
		0 & = x_{0}^{2} - x_{0} x_{2} + x_{0} x_{4} - x_{1}^{2} - x_{1} x_{3} + x_{1} x_{4} - x_{2}^{2} + x_{2} x_{3} - x_{2} x_{4},   \\
		0 & = x_{0} x_{1} - x_{0} x_{2} + x_{0} x_{4} - 2 x_{1}^{2} - x_{2}^{2} + x_{2} x_{3} - x_{2} x_{4} - x_{3}^{2} + x_{3} x_{4}, \\
		0 & = x_{0} x_{3} - x_{1}^{2} + x_{1} x_{2} + x_{1} x_{3} - x_{1} x_{4} + x_{2} x_{4}.
	\end{align*}
	By performing a brute-force search for rational solutions on this intersection,
	we identified a non-cuspidal point. To represent this point explicitly, we
	fixed the projective coordinates $x_{1} = -5$ and $x_{2} = -4$, which determined
	the integer representative $P = [3 : -5 : -4 : 2 : 9]$. Using the remaining
	three coordinates ($x_{0}, x_{3}, x_{4}$), we constructed an affine chart with
	variables $X = x_{3}/x_{0}$ and $Y = x_{4}/x_{0}$. The $q$-expansions of these
	affine coordinates are given by
	\begin{align*}
		X & = q^{-1}+ 2 + 3q + 6q^{2} + 6q^{3} + \cdots, \\
		Y & = q^{-2}+ q^{-1}+ 2 + 2q + 5q^{2} + \cdots.
	\end{align*}
	Evaluating these functions at the point $P$ yields the values
	$(X(P), Y(P)) = (2/3, 3)$.

	For the highest level $N=163$, the modular curve $X_{0}(163)$ has genus~13. We
first construct a plane model of $X_{0}(163)$ in order to apply
Algorithm~\ref{al:isogeny}.
	Following Galbraith's method with a basis $\{x_{0}, \dots, x_{12}\}$ of
	$S_{2}(\Gamma_{0}(163))$, we defined the plane coordinates $X = x_{1}/x_{11}$
	and $Y = x_{6}/x_{11}$. The $q$-expansions of these coordinates are:
	\begin{align*}
		\label{eq:qexp163}X & = q^{-2}+ q - 2q^{2} + q^{3} + q^{4} + \cdots, \\
		Y                   & = q^{-3}+ 1 - q + 2q^{2} - 3q^{4} + \cdots.
	\end{align*}
	We determined that $X$ and $Y$ satisfy a plane equation of total degree 16 with
	118 terms:
	\begin{equation}
		\label{eq:plane163}27X^{14}Y^{2}+ 22X^{13}Y^{3}+ 102X^{12}Y^{4}+ 121X^{11}Y^{5}
		+ \cdots + 4X^{2}Y^{4}= 0
	\end{equation}
	However, we encountered a computational obstruction in the final step of Algorithm
	\ref{al:isogeny}. The rational maps expressing the modular invariants $a_{4}$
	and $a_{6}$ in terms of $X$ and $Y$ involve polynomials of degree exceeding 40
	with astronomically large integer coefficients. This implies that even if a rational
	point $P$ were identified on the curve (\ref{eq:plane163}), substituting its
	coordinates into these maps would cause severe intermediate expression swell, rendering
	the direct algebraic evaluation computationally intractable.

	To circumvent this, we adopted an analytic approach. Since the unique non-cuspidal
	rational point on $X_{0}(163)$ corresponds to the CM field $\mathbb{Q}(\sqrt{-163}
	)$, we utilized the explicit Heegner point $\tau \in \mathbb{H}$:
	\[
		\tau = \frac{-163 + \sqrt{-163}}{326}.
	\]
	We evaluated the modular forms $a_{4}(\tau)$ and $a_{6}(\tau)$ directly at this
	$\tau$ with 4000 bits of precision (approx. $10^{-1200}$ error) and
	successfully recovered the exact rational invariants using rational
	reconstruction. The computed invariants for the domain curve $E$ and the codomain
	curve $E'$ are given by:
	\begin{equation}
		\label{eq:163_invariants}
		\begin{aligned}
			E  & : \quad a_{4} = -\frac{543605}{75481344}, \quad   &  & a_{6} = \frac{4936546769}{20985021333504}, \\
			E' & : \quad a'_{4} = -\frac{3335}{12303459072}, \quad &  & a'_{6} = -\frac{185801}{3420558477361152}.
		\end{aligned}
	\end{equation}
	We confirmed that both constructed curves belong to the canonical isogeny
	class \href{https://www.lmfdb.org/EllipticCurve/Q/26569.a1}{\textbf{26569.a1}}($j
	=-640320^{3}$). Specifically, $E$ and $E'$ are identified as the quadratic twists
	by the factors $D=4344$ and $D'=-708072$, respectively. The observed relation $D
	' = -163 \times D$ aligns perfectly with the theory of complex multiplication,
	reflecting that the cyclic 163-isogeny is induced by the endomorphism
	associated with $\sqrt{-163}$.

	Finally, as a consistency check of our plane model, we evaluated the
	coordinates $X$ and $Y$ at the same Heegner point $\tau$. Using a precision of
	$10^{-300}$, we recovered the simple rational coordinates:
	\[
		P = (X(\tau), Y(\tau)) = \left( \frac{9}{10}, -\frac{6}{5}\right).
	\]
	We verified that this point $P$ exactly satisfies the plane equation \eqref{eq:plane163},
	verifying the correctness of our algebraic model.

	By evaluating the modular generators $a_{4}$ and $a_{6}$ at the rational points
	$P \in X_{0}(N)(\mathbb{Q})$ on these specific models, we recovered the
	Weierstrass coefficients for the associated domain and codomain elliptic
	curves. Table \ref{tab:sporadic_isogenies} presents the complete classification
	results. For each rational point $P$ (expressed in the coordinates of the chosen
	model), we list the computed invariants $(a_{4}, a_{6})$ and $(a'_{4}, a'_{6})$.
	The resulting isogeny $E \to E'$ is identified by the LMFDB label $\mathbf{L}$ and an
	integer $D$, denoted as $\mathbf{L}^{(D)}$, representing the quadratic twist by $D$.
	{\small \setlength{\LTcapwidth}{\textwidth}
	\renewcommand{\arraystretch}{2.2} 
	\begin{longtable}{c|c|c|c} \caption{Explicit construction of sporadic $\mathbb{Q}$-rational cyclic $N$-isogenies.} \label{tab:sporadic_isogenies} \\

	\hline \textbf{Level} & \textbf{Point} $P$ & \textbf{Domain} $E$ and \textbf{Codomain} $E'$ & \textbf{Isogeny Class} \\ $N$ & $(X, Y)$ & $(a_{4}, a_{6} ; a'_{4}, a'_{6})$ & $E \to E'$ \\ \hline \endfirsthead

	\multicolumn{4}{c}%
	{{\bfseries \tablename\ \thetable{} -- continued from previous page}} \\ \hline \textbf{Level} & \textbf{Point} $P$ & \textbf{Domain} $E$ and \textbf{Codomain} $E'$ & \textbf{Isogeny Class} \\ $N$ & $(X, Y)$ & $(a_{4}, a_{6};a'_{4}, a'_{6})$ & $E \to E'$ \\ \hline \endhead

	\hline \multicolumn{4}{r}{ } \\ \endfoot

	\hline \endlastfoot

	\multirow{3}{*}{11} & $(5, 5)$ & $\left(-\frac{4323}{169}, -\frac{109406}{2197};-\frac{3}{169}, \frac{86}{24167}\right)$ & $\href{https://www.lmfdb.org/EllipticCurve/Q/121.a1}{\textbf{121.a1}}^{(39)}\to \href{https://www.lmfdb.org/EllipticCurve/Q/121.c1}{\textbf{121.c1}}^{(-429)}$ \\* & $(16, -61)$ & $\left(-\frac{363}{169}, -\frac{10406}{2197};-\frac{393}{1859}, \frac{9946}{265837}\right)$ & $\href{https://www.lmfdb.org/EllipticCurve/Q/121.c1}{\textbf{121.c1}}^{(39)}\to \href{https://www.lmfdb.org/EllipticCurve/Q/121.a1}{\textbf{121.a1}}^{(-429)}$ \\* & $(5, -6)$ & $\left(-\frac{33}{2}, -\frac{847}{32};-\frac{3}{22}, \frac{7}{352}\right)$ & $\href{https://www.lmfdb.org/EllipticCurve/Q/121.b1}{\textbf{121.b1}}^{(-486)}\to \href{https://www.lmfdb.org/EllipticCurve/Q/121.b1}{\textbf{121.b1}}^{(66)}$ \\ \hline

	\multirow{2}{*}{14} & $(2, 2)$ & $\left(-\frac{2380}{121}, -\frac{44688}{1331};-\frac{20}{847}, \frac{16}{9317}\right)$ & $\href{https://www.lmfdb.org/EllipticCurve/Q/49.a2}{\textbf{49.a2}}^{(22)}\to \href{https://www.lmfdb.org/EllipticCurve/Q/49.a1}{\textbf{49.a1}}^{(-154)}$ \\* & $(9, -33)$ & $\left(-\frac{560}{121}, -\frac{6272}{1331};-\frac{85}{847}, \frac{114}{9317}\right)$ & $\href{https://www.lmfdb.org/EllipticCurve/Q/49.a1}{\textbf{49.a1}}^{(11)}\to \href{https://www.lmfdb.org/EllipticCurve/Q/49.a2}{\textbf{49.a2}}^{(-77)}$ \\ \hline

	\multirow{4}{*}{15} & $(-2, -2)$ & $\left(3165, 31070;-3, \frac{118}{5}\right)$ & $\href{https://www.lmfdb.org/EllipticCurve/Q/50.b2}{\textbf{50.b2}}^{(-3)}\to \href{https://www.lmfdb.org/EllipticCurve/Q/50.a1}{\textbf{50.a1}}^{(-15)}$ \\* & $(3, -2)$ & $\left(-\frac{18075}{961}, -\frac{935350}{29791};-\frac{87}{4805}, \frac{842}{744775}\right)$ & $\href{https://www.lmfdb.org/EllipticCurve/Q/50.a2}{\textbf{50.a2}}^{(93)}\to \href{https://www.lmfdb.org/EllipticCurve/Q/50.b1}{\textbf{50.b1}}^{(465)}$ \\* & $\left(-\frac{13}{4}, \frac{9}{8}\right)$ & $\left(-675, -79650;\frac{211}{15}, -\frac{6214}{675}\right)$ & $\href{https://www.lmfdb.org/EllipticCurve/Q/50.a1}{\textbf{50.a1}}^{(1)}\to \href{https://www.lmfdb.org/EllipticCurve/Q/50.b2}{\textbf{50.b2}}^{(5)}$ \\* & $(8, -27)$ & $\left(-\frac{3915}{961}, -\frac{113670}{29791};-\frac{241}{2883}, \frac{37414}{4021785}\right)$ & $\href{https://www.lmfdb.org/EllipticCurve/Q/50.b1}{\textbf{50.b1}}^{(-31)}\to \href{https://www.lmfdb.org/EllipticCurve/Q/50.a2}{\textbf{50.a2}}^{(-155)}$ \\ \hline
	\multirow{2}{*}{17} & $\left(\frac{11}{4}, -\frac{15}{8}\right)$ & $\left(-\frac{87567}{5120}, -\frac{2230213}{81920};-\frac{1119}{87040}, \frac{14891}{23674880}\right)$ & $\href{https://www.lmfdb.org/EllipticCurve/Q/14450.b1}{\textbf{14450.b1}}^{(-30)}\to \href{https://www.lmfdb.org/EllipticCurve/Q/14450.b2}{\textbf{14450.b2}}^{(-510)}$ \\* & $(7, -21)$ & $\left(-\frac{19023}{5120}, -\frac{253147}{81920};-\frac{303}{5120}, \frac{7717}{1392640}\right)$ & $\href{https://www.lmfdb.org/EllipticCurve/Q/14450.b2}{\textbf{14450.b2}}^{(30)}\to \href{https://www.lmfdb.org/EllipticCurve/Q/14450.b1}{\textbf{14450.b1}}^{(510)}$ \\ \hline
	19 & $(5, -9)$ & $\left(-\frac{19}{2}, -\frac{361}{32};-\frac{1}{38}, \frac{1}{608}\right)$ & $\href{https://www.lmfdb.org/EllipticCurve/Q/361.a1}{\textbf{361.a1}}^{(-2)}\to \href{https://www.lmfdb.org/EllipticCurve/Q/361.a1}{\textbf{361.a1}}^{(38)}$ \\ \hline
	\multirow{4}{*}{21} & $(2, -1)$ & $\left(-\frac{17235}{1156}, -\frac{435447}{19652};-\frac{25}{3468}, \frac{131}{530604}\right)$ & $\href{https://www.lmfdb.org/EllipticCurve/Q/162.b4}{\textbf{162.b4}}^{(102)}\to \href{https://www.lmfdb.org/EllipticCurve/Q/162.b1}{\textbf{162.b1}}^{(102)}$ \\* & $(-1, 2)$ & $\left(-\frac{1515}{4}, -\frac{23053}{4};\frac{5}{4}, \frac{1}{12}\right)$ & $\href{https://www.lmfdb.org/EllipticCurve/Q/162.b3}{\textbf{162.b3}}^{(2)}\to \href{https://www.lmfdb.org/EllipticCurve/Q/162.b2}{\textbf{162.b2}}^{(2)}$ \\* & $\left(-\frac{1}{4}, \frac{1}{8}\right)$ & $\left(\frac{2205}{4}, -\frac{3087}{4};-\frac{505}{588}, \frac{23053}{37044}\right)$ & $\href{https://www.lmfdb.org/EllipticCurve/Q/162.b2}{\textbf{162.b2}}^{(-42)}\to \href{https://www.lmfdb.org/EllipticCurve/Q/162.b3}{\textbf{162.b3}}^{(-42)}$ \\* & $(5, -13)$ & $\left(-\frac{3675}{1156}, -\frac{44933}{19652};-\frac{1915}{56644}, \frac{48383}{20221908}\right)$ & $\href{https://www.lmfdb.org/EllipticCurve/Q/162.b1}{\textbf{162.b1}}^{(-238)}\to \href{https://www.lmfdb.org/EllipticCurve/Q/162.b4}{\textbf{162.b4}}^{(-238)}$ \\ \hline
	27 & $(3, -9)$ & $\left(-\frac{15}{2}, -\frac{253}{32};-\frac{5}{486}, \frac{253}{629856}\right)$ & $\href{https://www.lmfdb.org/EllipticCurve/Q/27.a2}{\textbf{27.a2}}^{(6)}\to \href{https://www.lmfdb.org/EllipticCurve/Q/27.a2}{\textbf{27.a2}}^{(-2)}$ \\ \hline

	\multirow{2}{*}{37} & $(0, -1)$ & $\left(-\frac{285371}{20580}, -\frac{180376009}{9075780};-\frac{11}{20580}, \frac{47}{9075780}\right)$ & $\href{https://www.lmfdb.org/EllipticCurve/Q/1225.h2}{\textbf{1225.h2}}^{(10)}\to \href{https://www.lmfdb.org/EllipticCurve/Q/1225.h1}{\textbf{1225.h1}}^{(10)}$ \\* & $[1:-1:0]$ & $\left(-\frac{15059}{20580}, -\frac{2380691}{9075780};-\frac{285371}{28174020}, \frac{180376009}{459715484340}\right)$ & $\href{https://www.lmfdb.org/EllipticCurve/Q/1225.h1}{\textbf{1225.h1}}^{(-370)}\to \href{https://www.lmfdb.org/EllipticCurve/Q/1225.h2}{\textbf{1225.h2}}^{(-370)}$ \\* \hline

	43& $\left(0, -\frac{4}{3}\right)$ & $\left(-\frac{215}{36}, -\frac{12943}{2304};-\frac{5}{1548}, \frac{7}{99072}\right)$ & $\href{https://www.lmfdb.org/EllipticCurve/Q/1849.a1}{\textbf{1849.a1}}^{(-3)}\to \href{https://www.lmfdb.org/EllipticCurve/Q/1849.a1}{\textbf{1849.a1}}^{(129)}$ \\* \hline

	67 & $\left(\frac{2}{3}, 3\right)$ & $\left(-\frac{3685}{722}, -\frac{974113}{219488};-\frac{55}{48374}, \frac{217}{14705696}\right)$ & $\href{https://www.lmfdb.org/EllipticCurve/Q/4489.a1}{\textbf{4489.a1}}^{(-38)}\to \href{https://www.lmfdb.org/EllipticCurve/Q/4489.a1}{\textbf{4489.a1}}^{(2546)}$ \\* \hline

	163 & $\left(\frac{9}{10}, -\frac{6}{5}\right)$ & $\text{See equation \eqref{eq:163_invariants}.}$ & $\href{https://www.lmfdb.org/EllipticCurve/Q/26569.a1}{\textbf{26569.a1}}^{(4544)}\to \href{https://www.lmfdb.org/EllipticCurve/Q/26569.a1}{\textbf{26569.a1}}^{(-708072)}$ \\\end{longtable} }
	\bibliographystyle{siam}
	\bibliography{bibliography1}

@article{Baaziz,
 author = {Baaziz, Houria},
 title = {Equations for the modular curve {{\(X_1(N)\)}} and models of elliptic curves with torsion points},
 fjournal = {Mathematics of Computation},
 journal = {Math. Comput.},
 issn = {0025-5718},
 volume = {79},
 number = {272},
 pages = {2371--2386},
 year = {2010},
 language = {English},
 doi = {10.1090/S0025-5718-10-02332-X},
 zbMATH = {5797915},
 Zbl = {1252.11038}
}

@manual{sagemath,
    label        = {Sag105},
    author       = {{The Sage Developers}},
    title        = {{S}age{M}ath, the {S}age {M}athematics {S}oftware {S}ystem},
    url          = {https://www.sagemath.org},
    version      = {10.5},
    year         = {2022},
    note         = {DOI 10.5281/zenodo.6259615},
}

@article{Yang,
 author = {Yang, Yifan},
 title = {Defining equations of modular curves},
 fjournal = {Advances in Mathematics},
 journal = {Adv. Math.},
 issn = {0001-8708},
 volume = {204},
 number = {2},
 pages = {481--508},
 year = {2006},
 language = {English},
 doi = {10.1016/j.aim.2005.05.019},
 zbMATH = {5047151},
 Zbl = {1137.11027}
}

@article{Reichert,
 author = {Reichert, Markus A.},
 title = {Explicit determination of nontrivial torsion structures of elliptic curves over quadratic number fields},
 fjournal = {Mathematics of Computation},
 journal = {Math. Comput.},
 issn = {0025-5718},
 volume = {46},
 pages = {637--658},
 year = {1986},
 language = {English},
 doi = {10.1090/S0025-5718-1986-0829635-X},
 zbMATH = {3977156},
 Zbl = {0605.14028}
}

@article {oesterle,
    AUTHOR = {Oesterl{\'e}, Joseph},
     TITLE = {Nombres de classes des corps quadratiques imaginaires},
      NOTE = {Seminar Bourbaki, Vol. 1983/84},
   JOURNAL = {Ast\'erisque},
  FJOURNAL = {Ast\'erisque},
    NUMBER = {121-122},
      YEAR = {1985},
     PAGES = {309--323},
      ISSN = {0303-1179},
   MRCLASS = {11R29 (11R11)},
MRREVIEWER = {Alan Candiotti},
}

@article {mazur78,
    AUTHOR = {Mazur, B.},
     TITLE = {Rational isogenies of prime degree (with an appendix by {D}.
              {G}oldfeld)},
   JOURNAL = {Invent. Math.},
  FJOURNAL = {Inventiones Mathematicae},
    VOLUME = {44},
      YEAR = {1978},
    NUMBER = {2},
     PAGES = {129--162},
      ISSN = {0020-9910},
     CODEN = {INVMBH},
   MRCLASS = {14K07 (10D35 14G25)},
MRREVIEWER = {V. V. Shokurov},
       DOI = {10.1007/BF01390348},
       URL = {http://dx.doi.org/10.1007/BF01390348},
}

@preamble{
   "\def\cprime{$'$} "
}

@book {silverman,
    AUTHOR = {Silverman, Joseph H.},
     TITLE = {The arithmetic of elliptic curves},
    SERIES = {Graduate Texts in Mathematics},
    VOLUME = {106},
   EDITION = {2nd},
 PUBLISHER = {Springer, Dordrecht},
      YEAR = {2009},
     PAGES = {xx+513},
      ISBN = {978-0-387-09493-9},
   MRCLASS = {11-02 (11G05 11G20 14H52 14K15)},
MRREVIEWER = {Vasil{\cprime} {\=I}. Andr{\={\i}}{\u\i}chuk},
       DOI = {10.1007/978-0-387-09494-6},
       URL = {http://dx.doi.org/10.1007/978-0-387-09494-6},
}

@book {modular,
    AUTHOR = {Diamond, Fred and Shurman, Jerry},
     TITLE = {A first course in modular forms},
    SERIES = {Graduate Texts in Mathematics},
    VOLUME = {228},
 PUBLISHER = {Springer-Verlag, New York},
      YEAR = {2005},
     PAGES = {xvi+436},
      ISBN = {0-387-23229-X},
   MRCLASS = {11Fxx},
MRREVIEWER = {Henri Darmon},
}

@unpublished {Dowd,
    AUTHOR = {Dowd, C. J. D},
     TITLE = {On parameterizations of cyclic $N$-isogenies and strict $K$-curves lying above rational points of $Y_0^+(N)$},
   NOTE = {preprint, submitted (http://arxiv.org/abs/2110.13908)}
}

@PHDTHESIS{Tsukazaki,
    AUTHOR = {K. Tsukazaki},
     TITLE = {Explicit Isogenies of Elliptic Curves},
     SCHOOL = {University of Warwick},
      YEAR = {2013},
}

@article{sutherland2,
 author = {Sutherland, Andrew V.},
 title = {Constructing elliptic curves over finite fields with prescribed torsion},
 fjournal = {Mathematics of Computation},
 journal = {Math. Comput.},
 issn = {0025-5718},
 volume = {81},
 number = {278},
 pages = {1131--1147},
 year = {2012},
 language = {English},
 doi = {10.1090/S0025-5718-2011-02538-X},
 zbMATH = {6028401},
 Zbl = {1267.11074}
}

@article {kenku,
    AUTHOR = {Kenku, M. A.},
     TITLE = {On the number of {${\bf Q}$}-isomorphism classes of elliptic
              curves in each {${\bf Q}$}-isogeny class},
   JOURNAL = {J. Number Theory},
  FJOURNAL = {Journal of Number Theory},
    VOLUME = {15},
      YEAR = {1982},
    NUMBER = {2},
     PAGES = {199--202},
      ISSN = {0022-314X},
   MRCLASS = {14K07 (14G25)},
  MRNUMBER = {675184},
MRREVIEWER = {Loren D. Olson},
       DOI = {10.1016/0022-314X(82)90025-7},
       URL = {https://doi.org/10.1016/0022-314X(82)90025-7},
}

@article {Galbraith-thesis,
    AUTHOR = {Galbraith, Steven D.},
     TITLE = {Equations for Modular Curves},
   JOURNAL = {Thesis (Ph.D.)},
    UNIVERSITY = {Oxford},
    YEAR = {1996},
}

@article{Jeon2025,
 author = {Jeon, Daeyeol},
 title = {Tetraelliptic modular curves {{\(X_1(N)\)}}},
 fjournal = {Acta Arithmetica},
 journal = {Acta Arith.},
 issn = {0065-1036},
 volume = {219},
 number = {1},
 pages = {33--51},
 year = {2025},
 language = {English},
 doi = {10.4064/aa231102-3-2},
 zbMATH = {8044226}
}

@misc{lmfdb,
  shorthand    = {LMFDB},
  author       = {{LMFDB Collaboration}},
  title        = {The {L}-functions and modular forms database},
  howpublished = {\url{http://www.lmfdb.org}},
  year         = {2025},
  note         = {[Online; accessed 13 February 2025]},
}

@misc{SageIsogenySmallDegree,
  author       = {The Sage Developers},
  title        = {Small Degree Isogenies of Elliptic Curves},
  howpublished = {\url{https://doc.sagemath.org/html/en/reference/arithmetic_curves/sage/schemes/elliptic_curves/isogeny_small_degree.html}},
  note         = {SageMath Documentation, accessed on 2025-11-29}
}

@article{Barrios,
 author = {Barrios, Alexander J.},
 title = {Explicit classification of isogeny graphs of rational elliptic curves},
 fjournal = {International Journal of Number Theory},
 journal = {Int. J. Number Theory},
 issn = {1793-0421},
 volume = {19},
 number = {4},
 pages = {913--936},
 year = {2023},
 language = {English},
 doi = {10.1142/S179304212350046X},
 zbMATH = {7661281},
 Zbl = {1518.11043}
}
\end{document}